\documentclass[preprint,12pt]{elsarticle}
\usepackage[margin=2.5cm]{geometry}

\usepackage{graphicx}
\usepackage{amssymb}

\usepackage[centertags]{amsmath}
\usepackage{amsfonts}
\usepackage{amssymb}
\usepackage{amsthm}
\usepackage[symbol*]{footmisc}
\DefineFNsymbolsTM{myfnsymbols}{
  \textasteriskcentered *
  \textdagger    \dagger
  \textdaggerdbl \ddagger
}
\setfnsymbol{myfnsymbols}

\newtheorem{thm}{Theorem}[section]
\newtheorem{prop}[thm]{Proposition}
\newtheorem{dfn}[thm]{Definition}

\newtheorem{oss}[thm]{Remark}

\def\RR{{\mathbb{R}}}

\def\NN{{\mathbb{N}}}
\newcommand{\eps}{\varepsilon}

\journal{Journal Name}

\begin{document}

\begin{frontmatter}


\title{Superexponential stabilizability of evolution equations of parabolic type via bilinear control\footnote{This paper was partly supported by the INdAM National Group for Mathematical Analysis, Probability and their Applications.}}



\author{Fatiha Alabau Boussouira}
\address{Laboratoire Jacques-Louis Lions Sorbonne Universit\'{e}, Universit\'{e} de Lorraine, 75005, Paris, France

alabau@ljll.math.upmc.fr}
\author{Piermarco Cannarsa\footnote{This author acknowledges support from the MIUR Excellence Department Project awarded to the Department of Mathematics, University of Rome Tor Vergata, CUP E83C18000100006.}} 
\address{Dipartimento di Matematica, Universit\`{a} di Roma Tor Vergata, 00133, Roma, Italy 

cannarsa@mat.uniroma2.it}
\author{Cristina Urbani\footnote{This author is grateful to University Italo Francese (Vinci Project 2018).}}
\address{Gran Sasso Science Institute, 67100, L'Aquila, Italy

Laboratoire Jacques-Louis Lions, Sorbonne Universit\'{e}, 75005, Paris, France

 cristina.urbani@gssi.it}

\begin{abstract}
We prove rapid stabilizability to the ground state solution for a class of abstract parabolic equations of the form
\begin{equation*}
u'(t)+Au(t)+p(t)Bu(t)=0,\qquad t\geq0
\end{equation*}
where the operator $-A$ is a self-adjoint accretive operator on a Hilbert space and $p(\cdot)$ is the control function. The proof is based on a linearization argument. We prove that the linearized system is exacly controllable and we apply the moment method to build a control $p(\cdot)$ that steers the solution to the ground state in finite time. Finally, we use such a control to bring the solution of the nonlinear equation arbitrarily close to the ground state solution with doubly exponential rate of convergence.

We give several applications of our result to different kinds of parabolic equations.
\end{abstract}


\end{frontmatter}


\section{Introduction}
In the field of control theory of dynamical systems a huge amount of works is devoted to the study of models in which the control enters as an additive term (boundary or internal locally distributed controls), see, for instance, the books \cite{lions1}, \cite{lions2} by J.L. Lions. On the other hand, these kinds of control systems are not suitable to describe processes that change their physical characteristics for the presence of the control action. This issue is quite common for the so-called \emph{smart materials} and in many biomedical, chemical and nuclear chain reactions. Indeed, under the process of \emph{catalysis} some materials are able to change their principal parameters (see the examples in \cite{k} for more details).

To deal with these situations, an important role in control theory is played by multiplicative controls that appear in the equations as coefficients. 

Due to a weaker control action, exact controllability results are not to be expected with multiplicative controls. On the other hand, approximate controllability has been obtained for different types of initial/target conditions. For instance, in \cite{k2} the author proved a result of non-negative approximate controllability of the $1D$ semilinear parabolic equation.  In \cite{k3}, the same author proved approximate and exact null controllability for a bilinear parabolic system with the reaction term satisfying Newton's law. Paper \cite{f} is devoted to the study of global approximate multiplicative controllability for nonlinear degenerate parabolic problems. In \cite{cfk} and \cite{ck}, results of approximate controllability of a one dimensional reaction-diffusion equation via multiplicative control and with sign changing data are proved.

An even more specific and weaker class of controls are bilinear controls which enter the equation as scalar functions of time as, for instance, in the following system
\begin{equation}\label{intr1}
\left\{
\begin{array}{ll}
u'(t)+Au(t)+p(t)Bu(t)=0,& t>0\\
u(0)=u_0.
\end{array}\right.
\end{equation}
A fundamental result in control theory for this type of evolution equations is the one due to Ball, Marsden and Slemrod \cite{bms} which establishes that the system \eqref{intr1} is not controllable. Indeed, if $u(t;p,u_0)$ denotes the unique solution of \eqref{intr1}, then the attainable set from $u_0$ defined by
\begin{equation*}
S(u_0)=\{ u(t;p,u_0);t\geq 0, p\in L^r_{loc}([0,+\infty),\RR),r>1\}
\end{equation*}
has a dense complement.

On the other hand, when $B$ is unbounded, the possibility of proving a positive controllability result remains open. This idea of exploiting the unboundness of the operator $B$ was developed by Beauchard and Laurent in \cite{bl} for the Schr{\"o}dinger equation
\begin{equation}\label{intr1.1}\left\{\begin{array}{ll}
iu_t(t,x)+u_{xx}(t,x)+p(t)\mu(x)u(t,x)=0,&(t,x)\in(0,T)\times(0,1)\\
u(t,0)=u(t,1)=0.
\end{array}\right.
\end{equation}
For such an equation the authors proved the local exact controllability around the ground state in a stronger topology than the natural one of $X=H^2\cap H^1_0(0,1)$ for which the multiplication operator $Bu(t,x)=\mu(x)u(t,x)$ is unbounded. In other terms, the above result could be regarded as a description of the attainable set from an initial submanifold of the original Banach space.

Following the same strategy, Beauchard in \cite{b} studied the wave equation
\begin{equation*}\left\{\begin{array}{ll}
u_{tt}(t,x)-u_{xx}(t,x)-p(t)\mu(x)u(t,x)=0,&(t,x)\in(0,T)\times(0,1)\\
u_x(t,0)=u(t,1)=0
\end{array}\right.
\end{equation*}
showing that for $T>2$ the system is locally controllable in a stronger topology than the natural one for this problem and for which the operator $Bu(t,x)=\mu(x)u(t,x)$ is unbounded.

In both papers \cite{b} and \cite{bl} a key point of the analysis is the application of the inverse mapping theorem which is made possible by the controllability of the linearized problem. This is the reason why, for parabolic problems, the above strategy meets an obstruction: the spaces for which one can prove controllability of the linearized equation are not well-adapted to the use of the inverse map technique.

We recall that some results of approximate controllability of hyperbolic equations with bilinear control have been achieved (see, for instance, \cite{cmsb}).

For other nonlinear equations in fluid dynamics, namely the Navier-Stokes equations with additive controls, the exact controllability to the uncontrolled solution of the equations was shown by Fern{\'a}ndez-Cara, Guerrero, Imanuvilov and Puel in \cite{fgip}.

In this paper, we are interested in studying the possibility of steering the solution of \eqref{intr1}, with a bilinear control, to a specific uncontrolled trajectory of the equation, namely the ground state solution.

To be more precise, let $X$ be a separable Hilbert space, $A:D(A)\subset X\to X$ be a self-adjoint accretive operator with compact resolvent (see section \ref{prelim} for more on notation and assumptions) and let $\{\lambda_k\}_{k\in\NN^*}$ be the eigenvalues of $A$, $(\lambda_k\leq\lambda_{k+1},\,\,\forall k\in\NN^*)$, with associated eigenfunctions $\{\varphi_k\}_{k\in\NN^*}$. Since it is customary to call $\varphi_1$ the \emph{ground state} of $A$, we refer to the function $\psi_1(t)=e^{-\lambda_1t}\varphi_1$ as the \emph{ground state solution}.

Our main result (Theorem \ref{ta1} below) ensures that, if $\{\lambda_k\}_{k\in\NN^*}$ satisfy a suitable gap condition (see condition \eqref{gap}) and $B$ spreads the ground state in all directions (see condition \eqref{a2}), then system \eqref{intr1} is locally stabilizable to $\psi_1$ at superexponential rate, that is, one can find a control $p\in L^2_{loc}(0,\infty)$ such that the corresponding solution $u(\cdot)$ of \eqref{intr1} satisfies
\begin{equation}\label{intr2}
\log{||u(t)-\psi_1(t)||}\leq C-e^{\omega t},\qquad\forall t>0,
\end{equation}
for suitable constants $C,\omega>0$.

An important point to underline is that our approach --- based on the moment method for the linearized system --- is fully constructive. First, we use the gap condition \eqref{gap} to build a biorthogonal family $\{\sigma_k(t)\}_{k\in\NN^*}$ to the exponentials $e^{\lambda_kt}$. Then, we apply such a family to construct a control $p(\cdot)$ that steers the linearized system of \eqref{intr1} exactly to the ground state solution in finite time. Finally, we repeatedly apply such exact controls for the linearized system in order to build a control $p(\cdot)$ for \eqref{intr1} which achieves \eqref{intr2}.

We point out that our method applies to both cases $\lambda_1=0$ and $\lambda_1>0$, giving an even faster decay rate in the latter case.

The above stabilizability result can be used to study several classes of parabolic problems, for which checking the validity of the assumptions on $A$ and $B$ is usually straightforward. For instance, we can treat the heat equation with a controlled source term of the form
\begin{equation*}
u_t-u_{xx}+p(t)\mu(x)u=0
\end{equation*}
with Dirichlet or Neumann boundary conditions, as well as operators with variable coefficients
\begin{equation*}
u_t-((1+x)^2u_x)_x+p(t)\mu(x)u=0,
\end{equation*}
or even $3D$ problems with radial data symmetry such as
\begin{equation*}
u_t-\Delta u+p(t)\mu(|x|)u=0.
\end{equation*}

Furthermore, we believe that the method we develop in this paper has potentials to be adapted to more general problems, such as a possibly unbounded operator $B$ and a degenerate principal part $A$.

The outline of the paper is the following. In section \ref{prelim}, we introduce the notation and the preliminary assumptions on the data. Section \ref{main} is devoted to our main result and its proof. Finally, in section \ref{examples} we give applications to several examples of parabolic problems.

\section{Preliminaries}\label{prelim}
Let $(X,\langle\cdot,\cdot\rangle)$ be a separable Hilbert space. We denote by $||\cdot||$ the associated norm on $X$.

Let $A:D(A)\subset X\to X$ be a densely defined linear operator with the following properties:
\begin{equation}\label{ip}
\begin{array}{ll}
(a) & A \mbox{ is self-adjoint},\\
(b) & A \mbox{ is accretive: }\langle Ax,x\rangle \geq 0,\,\, \forall x\in D(A),\\
(c) &\exists\,\lambda>0 \mbox{ such that }(\lambda I+A)^{-1}:X\to X \mbox{ is compact}.
\end{array}
\end{equation}

We recall that under the above assumptions $A$ is a closed operator and $D(A)$ is itself a Hilbert space with the scalar product
\begin{equation*} 
(x|y)_{D(A)}=\langle x,y\rangle+\langle Ax,Ay\rangle,\qquad\forall x,y \in D(A).
\end{equation*}
Moreover, $-A$ is the infinitesimal generator of a strongly continuous semigroup of contractions on $X$ which will be denoted by $e^{-tA}$. Furthermore, $e^{-tA}$ is analytic.

In view of the above assumptions, there exists an orthonormal basis $\{\varphi_k\}_{k\in\NN^*}$ in $X$ of eigenfunctions of $A$, that is, $\varphi_k\in D(A)$ and $A\varphi_k=\lambda_k\varphi_k$ $\forall k \in \NN^*$, where $\{\lambda_k\}_{k\in\NN^*}\subset \RR$ denote the corresponding eigenvalues. We recall that $\lambda_k\geq 0$, $\forall k\in\NN^*$ and we suppose --- without loss of generality --- that $\{\lambda_k\}_{k\in\NN^*}$ is ordered so that $0\leq\lambda_k\leq\lambda_{k+1}\to \infty$ as $k\to\infty$.
The associated semigroup has the following representation
\begin{equation}\label{semigr}
e^{-tA}\varphi=\sum_{k=1}^\infty\langle \varphi,\varphi_k\rangle e^{-\lambda_k t}\varphi_k,\quad\forall \varphi\in X.
\end{equation}

For any $s\geq 0$, we denote by $A^s:D(A^s)\subset X\to X$ the fractional power of $A$ (see \cite{p}). Under our assumptions, such a linear operator is characterized as follows
\begin{equation}
\begin{array}{l}
D(A^s)=\left\{x\in X ~\left|~ \sum_{k\in\NN^*}\lambda_k^{2s}|\langle x,\varphi_k\rangle|^2<\infty\right.\right\}\\\\
A^{s}x=\sum_{k\in\NN^*}\lambda_k^{s}\langle x,\varphi_k\rangle\varphi_k,\qquad \forall x\in D(A^s).
\end{array}
\end{equation}

Let $T>0$ and consider the problem
\begin{equation}
\left\{
\begin{array}{ll}\label{011}
u'(t)+Au(t)=f(t),&t\in[0,T]\\
u(0)=u_0
\end{array}
\right.
\end{equation}
where $u_0\in X$ and $f\in L^2(0,T;X)$. We now recall two definitions of solution of problem \eqref{011}:
\begin{itemize}
\item the function $u\in C([0,T], X)$ defined by $$u(t)=e^{-tA}u_0+\int_0^t e^{-(t-s)A}f(s)ds$$ is called the \emph{mild solution} of \eqref{011},
\item $u$ is a \emph{strong solution} of \eqref{011} in $L^2(0,T;X)$ if there exists a sequence $\{u_k\}\subseteq H^1(0,T;X)\cap L^2(0,T;D(A))$ such that
\begin{equation*}
\begin{array}{c}
u_k\to u, \mbox{  and  } u'_k-Au_k\to f \mbox{  in  }L^2(0,T;X),\\\\
u_k(0)\to u_0\mbox{  in  } X, \mbox{  as  } k\to \infty.
\end{array}
\end{equation*}
\end{itemize}

The well-posedness of the Cauchy problem \eqref{011} is a classical result (see, for instance, \cite{bd}).
\begin{thm}\label{wellpos}
Let $u_0\in X$ and $f\in L^2(0,T;X)$. Under hypothesis \eqref{ip}, problem \eqref{011} has a unique strong solution in $L^2(0,T;X)$. Moreover $u$ belongs to $C([0,T];X)$ and is given by the formula
\begin{equation}\label{013}
u(t)=e^{-tA}u_0+\int_0^t e^{-(t-s)A}f(s)ds.
\end{equation}
Furthermore, there exists a constant $C_0(T)>0$ such that
\begin{equation}\label{014}
\sup_{t\in[0,T]}||u(t)||\leq C_0(T)\left(||u_0||+||f||_{L^2(0,T;X)}\right)
\end{equation}
and $C_0(T)$ is non decreasing with respect to $T$.
\end{thm}

Given $T>0$, we consider the bilinear control problem
\begin{equation}\label{a1}\left\{
\begin{array}{ll}
u'(t)+Au(t)+p(t)Bu(t)=0,& t\in [0,T]\\
u(0)=u_0
\end{array}\right.
\end{equation}
where $u$ is the state variable and $p\in L^2(0,T)$ is the control function and the bilinear stabilizability problem
\begin{equation}\label{stab}\left\{
\begin{array}{ll}
u'(t)+Au(t)+p(t)Bu(t)=0,& t>0\\
u(0)=u_0
\end{array}\right.
\end{equation}
with $p\in L^2_{loc}([0,+\infty))$.

We recall that, in general, the exact controllability problem for system \eqref{a1} has a negative answer as shown by Ball, Marsden and Slemrod in \cite{bms}.


\section{Main result}\label{main}
We are interested in studying the stabilizability of system \eqref{stab} to a fixed trajectory. Let $X$ be a Hilbert space equipped with the scalar product $\langle\cdot,\cdot\rangle$. We denote by $||\cdot||=\sqrt{\langle\cdot,\cdot\rangle}$ the associated norm and by $B_R(\varphi)$ the open ball of radius $R>0$, centered in $\varphi\in X$. Given an initial condition $u_0\in X$ and a control $p\in L^2_{loc}([0,+\infty))$, we denote by $u(\cdot;u_0,p):[0,+\infty)\to X$ the corresponding solution of \eqref{stab}. 
\begin{dfn}
Given an initial condition $\bar{u}_0\in X$ and a control $\bar{p}\in L^2_{loc}([0,+\infty))$, we say that the control system \eqref{stab} is \emph{locally stabilizable to $\bar{u}(\cdot;\bar{u}_0,\bar{p})$} if there exists $\delta>0$ such that, for every $u_0\in B_\delta(\bar{u}_0)$, there exists a control $p\in L^2_{loc}([0,+\infty))$ such that
$$\lim_{t\to+\infty}||u(t;u_0,p)-\bar{u}(t;\bar{u}_0,\bar{p})||=0.$$
\end{dfn}
\begin{dfn}
Given an initial condition $\bar{u}_0\in X$ and a control $\bar{p}\in L^2_{loc}([0,+\infty))$, we say that the control system \eqref{stab} is \emph{locally exponentially stabilizable to $\bar{u}(\cdot;\bar{u}_0,\bar{p})$} if for any $\rho>0$, there exists $R(\rho)>0$ such that, for every $u_0\in B_{R(\rho)}(\bar{u}_0)$, there exists a control $p\in L^2_{loc}([0,+\infty))$ and a constant $M>0$ such that
$$||u(t;u_0,p)-\bar{u}(t;\bar{u}_0,\bar{p})||\leq M e^{-\rho t},\qquad\forall t>0.$$
\end{dfn}
\begin{dfn}
Given an initial condition $\bar{u}_0\in X$ and a control $\bar{p}\in L^2_{loc}([0,+\infty))$, we say that the control system \eqref{stab} is \emph{locally superexponentially stabilizable to $\bar{u}(\cdot;\bar{u}_0,\bar{p})$} if for any $\rho>0$ there exists $R(\rho)>0$ such that, for every $u_0\in B_{R(\rho)}(\bar{u}_0)$, there exists a control $p\in L^2_{loc}([0,+\infty))$ such that
$$||u(t;u_0,p)-\bar{u}(t;\bar{u}_0,\bar{p})||\leq M e^{-\rho e^{\omega t}},\qquad\forall t>0,$$
where $M,\omega>0$ are suitable constants depending only on $A$ and $B$.
\end{dfn}

For any $j\in\NN^*$ we set $\psi_j(t)=e^{-\lambda_j t}\varphi_j$ and we call $\psi_1$ the \emph{ground state solution}. Observe that $\psi_j$ solves \eqref{stab} with $p=0$ and $u_0=\varphi_j$. We shall study the superexponential stabilizability of \eqref{stab} to the trajectory $\psi_1$.

We observe that if there exists $\nu>0$ such that $\langle Ax,x\rangle\geq\nu||x||^2$, for all $x\in D(A)$, then the semigroup generated by $-A$ satisfies
\begin{equation*}
||e^{-tA}||\leq e^{-\nu t},\qquad \forall t>0.
\end{equation*}
If we consider any initial condition $u_0\in X$, then the evolution of the free dynamics with initial condition $u_0$ can be represented by the action of the semigroup, $u(t)=e^{-tA}u_0$. Therefore, one can prove easily that, when $A$ is strictly accretive, choosing the control $p=0$, system \eqref{stab} is locally exponentially stabilizable to the trajectory $\psi_1$. Indeed,
\begin{equation}
||u(t)-\psi_1(t)||=||e^{-tA}u_0-e^{-tA}\varphi_1||\leq e^{-\nu t}||u_0-\varphi_1||
\end{equation}
and this quantity tends to $0$ as $t$ goes to $+\infty$. 

On the contrary, in the general case of an accretive operator $A$, we do not have a straightforward choice of $p$ to deduce any stabilizability property of system \eqref{stab} to the ground state $\psi_1$. 

The novelty of our work is the construction of a control function $p$ that brings $u(t)$ arbitrary close to $\psi_1(t)$ in a very short time. Namely, we prove that \eqref{stab} is locally superexponentially stabilizable to the ground state solution. This can be seen as a weak version of the exact controllability to trajectories.

Let $B: X\to X$ be a bounded linear operator. From now on we denote by $C_B$ the norm of $B$ 
\begin{equation*}
C_B=\sup_{\varphi\in X,\,\,||\varphi||=1}||B\varphi||
\end{equation*}
and, without loss of generality, we suppose $C_B\geq1$.

We can now state our main result.

\begin{thm}\label{ta1}
Let $A:D(A)\subset X\to X$ be a densely defined linear operator satisfying hypothesis \eqref{ip} and suppose that there exists a constant $\alpha>0$ such that the eigenvalues of $A$ fulfill the gap condition
\begin{equation}\label{gap}
\sqrt{\lambda_{k+1}}-\sqrt{\lambda_k}\geq \alpha,\quad\forall k\in \NN^*.
\end{equation}
Let $B: X\to X$ be a bounded linear operator and let $\tau>0$ be such that
\begin{equation}\label{a2}
\begin{array}{l}
\langle B\varphi_1,\varphi_k\rangle\neq 0,\qquad\forall k\in\NN^*,\\\\
\displaystyle{\sum_{k\in\NN^*}\frac{e^{-2\lambda_k\tau}}{|\langle B\varphi_1,\varphi_k\rangle|^2}<+\infty.}
\end{array}
\end{equation}
Then, system \eqref{stab} is superexponentially stabilizable to $\psi_1$. 

Moreover, for every $\rho>0$ there exists $R_{\rho}>0$ such that any $u_0\in B_{R_\rho}(\varphi_1)$ admits a control $p\in L^2_{loc}([0,+\infty))$ such that the corresponding solution $u(\cdot; u_0,p)$ of \eqref{stab} satisfies
\begin{equation}
||u(t)-\psi_1(t)||\leq M e^{-(\rho e^{\omega t}+\lambda_1 t)},\qquad \forall t\geq 0,
\end{equation}
where $M$ and $\omega$ are positive constants depending only on $A$ and $B$.
\end{thm}

%

To prove Theorem \ref{ta1} we first start assuming that the first eigenvalue of $A$ is zero, $\lambda_1=0$, and we prove the local superexponential stabilizability of \eqref{stab} to the trajectory $\varphi_1$. Then, we will recover the general case from this one.

The proof of Theorem \ref{ta1} will be built through a series of propositions. The first result is the well-posedness of the problem
\begin{equation}\label{a1f}\left\{
\begin{array}{ll}
u'(t)+A u(t)+p(t)Bu(t)+f(t)=0,& t\in [0,T]\\
u(0)=u_0.
\end{array}\right.
\end{equation}
We introduce the following notation: 
\begin{equation*}\begin{array}{l}
||f||_{2,0}:=||f||_{L^2(0,T;X)},\qquad\forall\,f\in L^2(0,T;X)\\\\
||f||_{\infty,0}:=||f||_{C([0,T];X)}=\sup_{t\in [0,T]}||f(t)||,\qquad\forall\, f\in C([0,T];X).
\end{array}
\end{equation*}

\begin{prop}\label{propa24}
Let $T>0$. If $u_0\in X$, $p\in L^2(0,T)$ and $f\in L^2(0,T;X)$, then there exists a unique mild solution of \eqref{a1f}, i.e. a function $u\in C([0,T];X)$ such that the following equality holds in $X$ for every $t\in [0,T]$,

\begin{equation}
u(t)=e^{-tA }u_0-\int_0^t e^{-(t-s)A}[p(s)Bu(s)+f(s)]ds.
\end{equation}
Moreover, there exists a constant $C_1(T)>0$ such that 
\begin{equation}\label{a5}
||u||_{\infty,0}\leq C_1(T) (||u_0||+||f||_{2,0}).
\end{equation}
\end{prop}

Hereafter, we denote by $C$ a generic positive constant which may differ from line to line even if the symbol remains the same. Constants which play a specific role will be distinguished by an index i.e., $C_0$, $C_B$, \dots.

The proof of the existence of the mild solution of \eqref{a1f} is given in \cite{bms}. For what concerns the bound for the solution $u$ of \eqref{a1f}, it turns out that if $C_0(T)C_B||p||_{L^2(0,T)}\leq 1/2$, then we have inequality (\ref{a5}) with $C_1=C_2$ defined by
\begin{equation}\label{c1}
C_2:=2C_0(T).
\end{equation} 
Otherwise, to obtain \eqref{a5}, we proceed subdividing the interval $[0,T]$ into smaller subintervals for which $C_0(T)C_B||p||_{L^2}\leq1/2$ in all of them, and in this case the constant $C_1$ of inequality \eqref{a5} is defined by
\begin{equation}\label{cc1}
C_1=(1+N)(2C_0(T/N))^N,
\end{equation}
where $N$ is the number of subintervals.

Consider the system
\begin{equation}\label{sys}\left\{\begin{array}{ll}
u'(t)+A u(t)+p(t)Bu(t)=0,& t\in[0,T]\\
u(0)=u_0,
\end{array}\right.
\end{equation}
and the trajectory $\varphi_1$ that is a solution of \eqref{sys} when $p=0$, $u_0=\varphi_1$ and $\lambda_1=0$. Set $v:=u-\varphi_1$, we observe that $v$ is the solution of the following Cauchy problem
\begin{equation}\label{v}
\left\{\begin{array}{ll}
v'(t)+A v(t)+p(t)Bv(t)+p(t)B\varphi_1=0,&t\in[0,T]\\
v(0)=v_0=u_0-\varphi_1.
\end{array}\right.
\end{equation}

\begin{oss}\label{oss37}
Applying Theorem \ref{wellpos}, we find that $v\in C([0,T];X)$ is a mild solution of \eqref{v}, that is
\begin{equation}
v(t)=e^{-tA}v_0-\int_0^tp(s)e^{-(t-s)A}B(v(s)+\varphi_1)ds=V_0(t)+V_1(t),
\end{equation}
where 
\begin{equation*}
\begin{array}{l}
V_0(t):=e^{-tA}v_0,\\\\
V_1(t):=-\int_0^tp(s)e^{-(t-s)A}B(v(s)+\varphi_1)ds.
\end{array}
\end{equation*}
Since $p(\cdot)B(v(\cdot)+\varphi_1)\in L^2(0,T;X)$, we have that $V_1\in H^1(0,T;X)\cap L^2(0,T;D(A))$, while $V_0\in C^1((0,T];X)\cap C((0,T];D(A))$. Therefore, for every $\eps\in (0,T)$, $v\in H^1(\eps,T;X)$ and for almost every $t\in[\eps,T]$ the following equality holds
\begin{equation}\label{vae}
v'(t)+Av(t)+p(t)Bv(t)+p(t)B\varphi_1=0.
\end{equation}
\end{oss}

Showing the stabilizability of the solution $u$ of \eqref{sys} to the trajectory $\varphi_1$ is equivalent to proving the stabilizability to $0$ of system \eqref{v}: we have to prove that there exists $\delta>0$ such that, for every initial condition $v_0$ that satisfies $||v_0||\leq\delta$, there exists a trajectory-control pair $(v,p)$ such that $\lim_{t\to +\infty}||v(t)||=0$.

For this purpose, we consider the following linearized system
\begin{equation}\label{lin}\left\{\begin{array}{ll}
\bar{v}(t)'+A \bar{v}(t)+p(t)B\varphi_1=0,&t\in[0,T]\\
\bar{v}(0)=v_0.
\end{array}\right.
\end{equation}
For this linear system we are able to prove the following null controllability result.
\begin{prop}\label{prop34}
Let $T>\tau$ and let $A$ and $B$ be such that \eqref{ip}, \eqref{gap}, \eqref{a2} hold and furthermore we assume $\lambda_1=0$. Let $v_0\in X$. Then, there exists a control $p\in L^2(0,T)$ such that $\bar{v}(T)=0$.

Moreover, there exists a constant $C_\alpha(T)>0$ such that
\begin{equation}\label{pbound}
||p||_{L^2(0,T)}\leq C_\alpha(T)\Lambda_T||v_0||
\end{equation}
where $\Lambda_T$ is defined in \eqref{lambdaT} and $\alpha>0$ is the constant in \eqref{gap}.
\end{prop}

Let us recall the notion of \emph{biorthogonal family} and a result we will use to show the null controllability of the linearized system \eqref{lin}.

\begin{dfn}
Let $\{\zeta_j\}$ and $\{\sigma_k\}$ be two sequences in a Hilbert space H. We say that the two families are \emph{biorthogonal} or that $\{\zeta_j\}$ (resp.$\{\sigma_k\}$) is \emph{biorthogonal} to $\{\sigma_k\}$ (resp. $\{\zeta_j\}$) if
$$\langle \zeta_j,\sigma_k\rangle_H=\delta_{j,k},\quad\forall j,k\geq 0$$
where $\delta_{j,k}$ is the Kronecker delta.
\end{dfn}

The notion of biorthogonal family was used by Fattorini and Russell in \cite{fr}, where they introduced the moment method. Such a technique was developed later by several authors. We recall below the result proved in \cite{cmv}.

\begin{thm}\label{t210}
Let $\{\omega_k\}_{k\in\NN}$ be an increasing sequence of nonnegative real numbers.
Assume that there exists a constant $\alpha>0$ such that
$$\forall k\in\NN,\quad \sqrt{\omega_{k+1}}-\sqrt{\omega_k}\geq\alpha.$$
Then, there exists a family $\{\sigma_j\}_{j\geq 0}$ which is biorthogonal to the family $\{e^{\omega_kt}\}_{k\geq 0}$ in $L^2(0,T)$, that is,
$$\forall k,j\in\NN ,\quad \int_0^T \sigma_j(t)e^{\omega_kt}dt=\delta_{jk}.$$

Furthermore, there exist two constants $C_\alpha,C_\alpha(T)>0$ such that
\begin{equation}\label{019}
||\sigma_j||^2_{L^2(0,T)}\leq C^2_\alpha(T) e^{-2\omega_j T}e^{C_\alpha\sqrt{\omega_j}/\alpha},\quad\forall j\in\NN.
\end{equation}
\end{thm}

\begin{oss}\label{oss35}
For all $T\in\RR$ we define the quantity
\begin{equation}\label{lambdaT}
\Lambda_T:=\left( \sum_{k\in\NN^*}\frac{e^{-2\lambda_kT}e^{C_{\alpha}\sqrt{\lambda_k}/\alpha}}{|\langle B\varphi_1,\varphi_k\rangle|^2}\right)^{1/2}
\end{equation}
and we observe that if there exists $\tau>0$ such that \eqref{a2} holds then, for every $T>\tau$, $\Lambda_T<+\infty$.

Furthermore, if $\lambda_1>0$ then $\Lambda_T\to 0$ as $T\to +\infty$.
\end{oss}

Thanks to Theorem \ref{t210} and Remark \ref{oss35} we are able to prove Proposition \ref{prop34}:

\begin{proof}[Proof (of Proposition \ref{prop34}).]
For any $v_0\in X$ and $p\in L^2(0,T)$, it follows from Proposition \ref{propa24} that there exists a unique mild solution $\bar{v}\in C^0([0,T],X)$ of \eqref{lin} that can be represented by the formula
\begin{equation}
\bar{v}(t)=e^{-tA}v_0-\int_0^t e^{-(t-s)A}p(s)B\varphi_1ds.
\end{equation}
We want to find $p\in L^2(0,T)$ such that $\bar{v}(T)=0$, thus the following equality must hold
\begin{equation}
\sum_{k\in\NN^*}\langle v_0,\varphi_k\rangle e^{-\lambda_k T}\varphi_k=\int_0^T p(s)\sum_{k\in\NN^*}\langle B\varphi_1,\varphi_k\rangle e^{-\lambda_k(T-s)}\varphi_kds.
\end{equation}
Since $\{\varphi_k\}_{k\in\NN^*}$ is an orthonormal basis of the space $X$, the equality must hold in every direction and it follows that
\begin{equation}
\langle v_0,\varphi_k\rangle=\int_0^T e^{\lambda_ks}p(s)\langle B\varphi_1,\varphi_k\rangle ds
\end{equation}
for every $k\in\NN^*$. Therefore, proving null controllability of the linearized system reduces to finding a function $p\in L^2(0,T)$ that satisfies
\begin{equation}\label{momp}
\int_0^T e^{\lambda_ks}p(s)ds=\frac{\langle v_0,\varphi_k\rangle}{\langle B\varphi_1,\varphi_k\rangle}
\end{equation}
for all $k\in\NN^*$.
Thanks to assumption (\ref{gap}), there exists $\alpha>0$ such that  the gap condition $\sqrt{\lambda_{k+1}}-\sqrt{\lambda_k}\geq\alpha$ holds for all $k\in\NN^*$. Then, Theorem \ref{t210} ensures the existence of a family $\{\sigma_k\}_{k\in\NN^*}$ that is biorthogonal to $\{e^{\lambda_ks}\}_{k\in\NN^*}$. Taking $p(s)=\sum_{k\in\NN^*}c_k\sigma_k(s)$ one finds that the coefficients $c_k$ are given by $c_k=\frac{\langle v_0,\varphi_k\rangle}{\langle B\varphi_1,\varphi_k\rangle}$, $\forall k\in\NN^*$. Thus, in order to show that
\begin{equation}\label{p}
p(s):=\sum_{k\in\NN^*}\frac{\langle v_0,\varphi_k\rangle}{\langle B\varphi_1,\varphi_k\rangle}\sigma_k(s)
\end{equation}
is a solution of \eqref{momp}, it suffices to prove that the series is convergent in $L^2(0,T)$. Indeed,
\begin{equation*}
||p||_{L^2(0,T)}\leq\sum_{k\in\NN^*}\left|\frac{\langle v_0,\varphi_k\rangle}{\langle B\varphi_1,\varphi_k\rangle}\right|||\sigma_k||_{L^2(0,T)}\leq ||v_0||\left(\sum_{k\in\NN^*}\frac{||\sigma_k||^2_{L^2(0,T)}}{|\langle B\varphi_1,\varphi_k\rangle|^2}\right)^{1/2}
\end{equation*}
and we appeal to estimate (\ref{019}) for $\{\sigma_k\}_{k\in\NN^*}$, with $\omega_k=\lambda_k$ for all $k\in\NN^*$, to obtain that
\begin{equation*}
\left(\sum_{k\in\NN^*}\frac{||\sigma_k||^2_{L^2(0,T)}}{|\langle B\varphi_1,\varphi_k\rangle|^2}\right)^{1/2}\leq\left( C^2_\alpha(T)\sum_{k\in\NN^*}\frac{e^{-2\lambda_kT}e^{C_\alpha\sqrt{\lambda_k}/\alpha}}{|\langle B\varphi_1,\varphi_k\rangle|^2})\right)^{1/2}=C_\alpha(T)\Lambda_T
\end{equation*}
that is finite thanks to hypothesis \eqref{a2} and Remark \ref{oss35}. Thus, the following bound for the $L^2$-norm of $p$ holds true:
\begin{equation*}
||p||_{L^2(0,T)}\leq C_\alpha(T)\Lambda_T||v_0||.
\end{equation*}
\end{proof}

In Proposition \ref{prop34} we have found a control $p$ that steers the solution of the linearized system to $0$ in time $T$. We use such a control in the nonlinear system \eqref{v} to obtain a uniform estimate for the solution $v(t)$.

\begin{prop}\label{prop38}
Let $A$ and $B$ satisfying hypotheses \eqref{ip}, \eqref{gap}, \eqref{a2} and furthermore we assume $\lambda_1=0$. Let $p\in L^2(0,T)$ be defined by the following formula
\begin{equation}\label{pdef}
p(t)=\sum_{k\in\NN^*}\frac{\langle v_0,\varphi_k\rangle}{\langle B\varphi_1,\varphi_k\rangle}\sigma_k(t)
\end{equation}
where $\{\sigma_k\}_{k\in\NN^*}$ is the biorthogonal family to $\{e^{\lambda_kt}\}_{k\in\NN^*}$ given by Theorem \ref{t210}.

Then, the solution $v$ of \eqref{v} satisfies
\begin{equation}\label{unifv}
\sup_{t\in[0,T]}||v(t)||^2\leq e^{C_3(T)\Lambda_T||v_0||+C_BT}(1+C_4(T)\Lambda_T^2)||v_0||^2
\end{equation}
where $C_B\geq1$ is the norm of the operator $B$, $C_3(T):=2\sqrt{T}C_BC_\alpha(T)$, and $C_4(T):=C_BC_\alpha^2(T)$.
\end{prop}
\begin{proof}
We consider the equation in \eqref{v}. Thanks to Remark \ref{oss37}, since \eqref{vae} is satisfied for almost every $t\in[\eps,T]$, we are allowed to take the scalar product with $v$:
\begin{equation}
\langle v'(t),v(t)\rangle+\langle A v(t),v(t)\rangle+p(t)\langle Bv(t)+B\varphi_1,v(t)\rangle=0.
\end{equation}
Thus, using that $B$ is bounded, we get
\begin{equation}\label{energy}
\begin{split}
\frac{1}{2}\frac{d}{dt}||v(t)||^2+\langle A v(t),v(t)\rangle&\leq  C_B\left(|p(t)|||v(t)||^2+|p(t)|||\varphi_1||||v(t)||\right)\\
&\leq C_B\left(|p(t)|||v(t)||^2+\frac{1}{2}|p(t)|^2+\frac{1}{2}||v(t)||^2\right)\\
\end{split}
\end{equation}
and therefore, since $A$ is accretive, we have that
\begin{equation*}
\frac{1}{2}\frac{d}{dt}||v(t)||^2\leq C_B\left(|p(t)|+\frac{1}{2}\right)||v(t)||^2+\frac{1}{2}C_B|p(t)|^2.
\end{equation*}
We integrate the last inequality from $\eps$ to $t$:
\begin{equation*}
\int_\eps^t \frac{d}{ds}||v(s)||^2ds\leq 2C_B\int_\eps^t\left(|p(s)|+\frac{1}{2}\right)||v(s)||^2ds+C_B\int_0^T|p(s)|^2ds
\end{equation*}
and by Gronwall's inequality, we obtain
\begin{equation*}
||v(t)||^2\leq\left(||v(\eps)||^2+C_B\int_0^T|p(s)|^2ds\right)e^{2C_B\int_\eps^t(|p(s)|+1/2)}ds
\end{equation*}
and taking the limit $\eps\to0$ we find that
\begin{equation*}
||v(t)||^2\leq\left(||v_0||^2+C_B\int_0^T|p(s)|^2ds\right)e^{2C_B\int_0^t(|p(s)|+1/2)}ds.
\end{equation*}
Thus, taking the supremum over the interval $[0,T]$, the last inequality becomes
\begin{equation}
\begin{split}
\sup_{t\in[0,T]}||v(t)||^2\leq e^{C_B\left(2\sqrt{T}||p||_{L^2(0,T)}+T\right)}\left(||v_0||^2+C_B||p||^2_{L^2(0,T)}\right)
\end{split}
\end{equation}
and finally, recalling the estimate \eqref{pbound} for the $L^2$-norm of $p$ from Proposition \ref{prop34}, we get
\begin{equation}
\sup_{t\in[0,T]}||v(t)||^2\leq e^{C_B\left(2\sqrt{T}C_\alpha(T)\Lambda_T||v_0||+T\right)}\left(1+C_BC^2_\alpha(T)\Lambda_T^2\right)||v_0||^2.
\end{equation}
\end{proof}

We want now to measure the distance at time $T$ of the solutions of the nonlinear system and the linearized one when using the same control function $p$ built by solving of the moment problem in Proposition \ref{prop34}.

Therefore, we introduce the function $w(t):=v(t)-\bar{v}(t)$ that satisfies the following Cauchy problem
\begin{equation}\label{w}
\left\{\begin{array}{ll}
w'(t)+Aw(t)+p(t)Bv(t)=0,&t\in[0,T]\\
w(0)=0.
\end{array}\right.
\end{equation}
We define the constant $K^2_T:=C_BC_4(T)\Lambda^2_Te^{C_3(T)+(C_B+1)T}(1+C_4(T)\Lambda_T^2)$.

\begin{prop}\label{prop39}
Let $A$ and $B$ satisfy hypotheses \eqref{ip}, \eqref{gap}, \eqref{a2}, and furthermore we assume $\lambda_1=0$. Let $T>\tau$, $p$ be defined by \eqref{pdef}, and let $v_0\in X$ be such that 
\begin{equation}\label{K_T}
K_T||v_0||\leq 1.
\end{equation}
Then, it holds that
\begin{equation}\label{wT}
||w(T)||\leq K_T||v_0||^2.
\end{equation}
\end{prop}

\begin{proof}
Observe that $w\in C([0,T];X)$ is the mild solution of \eqref{w}. Moreover $w\in H^1(0,T;X)\cap L^2(0,T;D(A))$ and thus $w$ satisfies the equality
\begin{equation}\label{eqw}
w'(t)+Aw(t)+p(t)Bv(t)=0
\end{equation}
for almost every $t\in [0,T]$.

We multiply equation \eqref{eqw} by $w(t)$ and we obtain
\begin{equation}\begin{split}
\frac{1}{2}\frac{d}{dt}||w(t)||^2&\leq |p(t)|||Bv(t)||||w(t)||\\
&\leq \frac{1}{2}||w(t)||^2+C^2_B\frac{1}{2}|p(t)|^2||v(t)||^2.
\end{split}
\end{equation}
Therefore, applying Gronwall's inequality, taking the supremum over $[0,T]$ and using \eqref{unifv} and \eqref{pbound}, we get
\begin{equation}\begin{split}
\sup_{t\in[0,T]}||w(t)||^2&\leq C^2_Be^T||p||^2_{L^2(0,T)}\sup_{t\in[0,T]}||v(t)||^2\\
&\leq C^2_Be^{C_3(T)\Lambda_T||v_0||+C_BT+T}(1+C_4(T)\Lambda_T^2)||v_0||^2||p||^2_{L^2(0,T)}\\
&\leq C^2_BC^2_\alpha(T)\Lambda_T^2e^{C_3(T)\Lambda_T||v_0||+(C_B+1)T}(1+C_4(T)\Lambda_T^2)||v_0||^4.
\end{split}
\end{equation}
We can suppose, without loss of generality, that $C_\alpha(T)\geq 1$. Thus, from \eqref{K_T}, we obtain that $\Lambda_T||v_0||\leq 1$. Therefore, 
\begin{equation*}
\sup_{t\in[0,T]}||w(t)||^2\leq K^2_T||v_0||^4,
\end{equation*}
that implies
\begin{equation}\label{vT}
||w(T)||\leq K_T||v_0||^2.
\end{equation}
\end{proof}

Recalling that $\bar{v}(T)=0$, we deduce from \eqref{wT} that 
\begin{equation}\label{vT}
||v(T)||\leq K_T||v_0||^2,
\end{equation}
and, moreover,
\begin{equation}\label{magg}
K_T||v(T)||\leq \left(K_T||v_0||\right)^2\leq 1.
\end{equation}
We observe that we can apply Proposition \ref{prop39} to problem \eqref{v} defined in the interval $[T,2T]$. Indeed, $v_T:=v(T)$ that was computed by solving \eqref{v}, is the initial condition of the problem
\begin{equation}\label{T2T}
\left\{\begin{array}{ll}
v_t(t)+A v(t)+p(t)Bv(t)+p(t)B\varphi_1=0,&t\in [T,2T]\\
v(T)=v_T.
\end{array}\right.
\end{equation}
We shift this problem to the interval $[0,T]$ by introducing the variable $s:=t-T$ in the above system.
If we set $\tilde{v}(s):=v(s+T)$ and $\tilde{p}:=p(s+T)$, then $\tilde{v}$ solves
\begin{equation}
\left\{\begin{array}{ll}
\tilde{v}_t(s)+A \tilde{v}(s)+\tilde{p}(s)B\tilde{v}(s)+\tilde{p}(s)B\varphi_1=0,&s\in [0,T]\\
\tilde{v}(0)=v_T.
\end{array}\right.
\end{equation}
Here the control $\tilde{p}$ is given by Proposition \ref{prop38}, with initial condition $v_T$, that is:
\begin{equation}
\tilde{p}(s)=\sum_{k\in\NN^*}\frac{\langle v_T,\varphi_k\rangle}{\langle B\varphi_1,\varphi_k\rangle}\sigma_k(s)
\end{equation}
where $\{\sigma_k(s)\}_{k\in\NN^*}$ is the biorthogonal family to $\{e^{\lambda_ks}\}_{k\in\NN^*}$ in $[0,T]$. Thus, it is possible to bound the $L^2$-norm of $\tilde{p}$ by
\begin{equation}
||\tilde{p}||_{L^2(0,T)}\leq C_\alpha(T)\Lambda_T||v_T||
\end{equation}
thanks to the estimate for $\{\sigma_k(s)\}_{k\in\NN^*}$ given in Theorem \ref{t210}.
Therefore, for the control $p$ of the linearized system associated to \eqref{T2T}, it holds that
\begin{equation*}
||p||_{L^2(T,2T)}=||\tilde{p}||_{L^2(0,T)}\leq C_\alpha(T)\Lambda_T||v_T||.
\end{equation*}

Finally, thanks to \eqref{magg}, the hypotheses of Proposition \ref{prop39} for problem \eqref{T2T} are satisfied and we obtain that $||v(2T)||\leq K_T||v(T)||^2$. Furthermore,
\begin{equation}
K_T||v(2T)||\leq(K_T||v_0||)^2\leq 1,
\end{equation}
and we can repeat this argument for the next intervals $[2T,3T], [3T,4T],\dots,[(n-1)T,nT],\dots$.
Therefore, we deduce that
\begin{equation}\label{maggn}
K_T||v(nT)||\leq 1, \qquad \forall n\in\NN^*.
\end{equation}

Now, we want to obtain an estimate as \eqref{vT} for the solution $v$ of problem \eqref{v} defined in time intervals of the form $[nT,(n+1)T]$, with $n\geq 1$.
\begin{prop}
Let $A$ and $B$ satisfy hypotheses \eqref{ip}, \eqref{gap}, \eqref{a2} and furthermore we assume $\lambda_1=0$. Let $v_0\in X$ be such that 
\begin{equation}
K_T||v_0||\leq 1.
\end{equation}
Then, the following iterated estimate holds:
\begin{equation}\label{vnT}
||v(nT)||\leq \frac{1}{K_T}\left( K_T||v_0||\right)^{2^n},\qquad \forall n\geq 0.
\end{equation}
\end{prop}

\begin{proof}
We proceed by induction on $n$. For $n=1$, the formula has been proved in Proposition \ref{prop39}. We suppose that \eqref{vnT} holds and we prove the estimate for $v((n+1)T)$:
iterating the construction of the solution $v$ of \eqref{v} in consecutive time intervals of the form $[kT,(k+1)T]$ until $k+1=n$, we come to the following problem
\begin{equation}\label{sysnT}
\left\{\begin{array}{ll}
v'(t)+A v(t)+p(t)Bv(t)+p(t)B\varphi_1=0, & t\in [nT,(n+1)T],\\
v(nT)=v_{nT}.
\end{array}
\right.
\end{equation}
where $v_{nT}$ is the value assumed at time $nT$ by the solution of the same problem solved in the interval $[(n-1)T,nT]$ with initial data $v_{(n-1)T}$. We shift this problem in the time interval $[0,T]$ by introducing the variable $s:=t-nT$ and the functions $\tilde{v}(s)=v(s+nT)$, $\tilde{p}(s)=p(s+nT)$. Then, $\tilde{v}$ is the solution of the following Cauchy problem
\begin{equation}\label{sysTtilde}
\left\{\begin{array}{ll}
\tilde{v}_t(s)+A \tilde{v}(s)+\tilde{p}(s)B\tilde{v}(s)+\tilde{p}(s)B\varphi_1=0,&s\in [0,T]\\
\tilde{v}(0)=v_{nT}.
\end{array}\right.
\end{equation}
The control function $\tilde{p}$ is defined in $[0,T]$ by solving the null controllability problem for the associated linearized system and its $L^2$-norm can be bound by
\begin{equation*}
||\tilde{p}||_{L^2(0,T)}\leq C_\alpha(T)\Lambda_T||v_{nT}||.
\end{equation*}
Therefore, coming back to the original time interval $[nT,(n+1)T]$ we find that
\begin{equation}\label{boundp}
||p||_{L^2(nT,(n+1)T)}=||\tilde{p}||_{L^2(0,T)}\leq C_\alpha(T)\Lambda_T||v_{nT}||.
\end{equation}
Moreover, since it holds that
\begin{equation}
K_T||v_{nT}||\leq 1
\end{equation}
we can use Proposition \ref{prop39} for problem \eqref{sysnT}, obtaining
\begin{equation}
||v((n+1)T)||\leq K_T||v(nT)||^2\leq K_T\left(\frac{1}{K_T}\left(K_T||v_0||\right)^{2^n}\right)^2=\frac{1}{K_T}\left(K_T||v_0||\right)^{2^{n+1}}
\end{equation}
and this concludes the induction argument and the proof of the proposition.
\end{proof}

The last step that allows us to prove Theorem \ref{ta1} consists in showing the rapid decay of the solution $u$ of our initial problem \eqref{stab} to the fixed stationary trajectory $\varphi_1$.

\begin{prop}\label{prop314}
Let $\theta\in(0,1)$ and $||v_0||\leq \frac{\theta}{K_T}$. Then, under the hypotheses \eqref{ip}, \eqref{gap}, \eqref{a2} and $\lambda_1=0$, there exists a constant $C_T>0$ such that
\begin{equation}
||u(t)-\varphi_1||\leq \frac{C_T}{K_T}\theta^{2^{t/T-1}}\qquad \forall t\geq0.
\end{equation}
\end{prop}

\begin{proof}
We have supposed that $||v_0||\leq \frac{\theta}{K_T}$, with $\theta\in(0,1)$. Thus, \eqref{vnT} becomes
\begin{equation}
||v(nT)||\leq \frac{\theta^{2^n}}{K_T}.
\end{equation}
Consider now the time interval $[nT,(n+1)T]$. From estimate \eqref{a5} for the solution of the control system in the time interval $[nT,(n+1)T]$ and from the bound \eqref{boundp} for the control $p$, we deduce that there exists a constant $C_T>0$ such that
\begin{equation}\label{c_T}
||v(t)||\leq C_T||v(nT)||,\qquad t\in[nT, (n+1)T].
\end{equation}
Therefore, using \eqref{vnT} in \eqref{c_T}, we obtain that
\begin{equation}
||v(t)||\leq C_T||v(nT)||\leq \frac{C_T}{K_T}\theta^{2^n}=\frac{C_T}{K_T}\left( \theta^{2^{(n+1)}}\right)^{1/2}.
\end{equation}
Since $n\leq \frac{t}{T}\leq (n+1)$ and $\theta\in (0,1)$, it holds that
\begin{equation}
||v(t)||\leq\frac{C_T}{K_T}\left( \theta^{2^{(n+1)}}\right)^{1/2}\leq \frac{C_T}{K_T}\left(\theta^{2^{t/T}} \right)^{1/2}=\frac{C_T}{K_T}\theta^{2^{t/T-1}}.
\end{equation}
By definition, $v(t)=u(t)-\varphi_1$. So, we get
\begin{equation}
||u(t)-\varphi_1||\leq \frac{C_T}{K_T}\theta^{2^{t/T-1}},\qquad t\geq0.
\end{equation}
\end{proof}

We are ready to prove Theorem \ref{ta1}.

\begin{proof}[Proof of Theorem \ref{ta1}]
We first consider the case in which the first eigenvalue of $A$ is zero.

Let $\theta\in (0,1)$ and let $\rho>0$ be the value for which $\theta=e^{-2\rho}$. Then, from Proposition \ref{prop314}, there exist a constant $R_{\rho}>0$ such that if $||u_0-\varphi_1||\leq R_{\rho}$, then 
\begin{equation*}
||u(t)-\varphi_1||\leq M_Te^{-\rho e^{\omega_T t}}, \forall t\geq0.
\end{equation*}
where $M_T,\omega_T>0$ are constants that depend only on $T$.
With the notation of the previous propositions, we have that
\begin{equation}
R_\rho:=\frac{e^{-2\rho}}{K_T},\qquad M_T:=\frac{C_T}{K_T},\qquad \omega_T:=\frac{\log{2}}{T}.
\end{equation}

Now, in order to deal with a general operator $A$ satisfying \eqref{ip}, we introduce the operator
\begin{equation}\label{A1}
A_1:=A-\lambda_1 I.
\end{equation}
We observe that $A_1:D(A_1)\subset X\to X$ is self-adjoint, accretive and $-A_1$ generates a strongly continuous analytic semigroup of contraction. Its eigenvalues are given by
\begin{equation}
\mu_k=\lambda_k-\lambda_1,\qquad \forall k\in\NN^*
\end{equation}
(in particular, $\mu_1=0$) and it has the same eigenfunctions as $A$, $\{\varphi_k\}_{k\in\NN^*}$. Moreover, the family $\{\mu_k\}_{k\in\NN^*}$ satisfies the same gap condition \eqref{gap} that is satisfied by the eigenvalues of $A$. Indeed, it holds that
\begin{equation*}
\sqrt{\mu_{k+1}}-\sqrt{\mu_k}=\frac{\lambda_{k+1}-\lambda_k}{\sqrt{\mu_{k+1}}+\sqrt{\mu_k}}\geq \frac{\lambda_{k+1}-\lambda_k}{\sqrt{\lambda_{k+1}}+\sqrt{\lambda_k}}=\sqrt{\lambda_{k+1}}-\sqrt{\lambda_k}\geq\alpha,\qquad \forall k\in\NN^*.
\end{equation*}
Thus, the operator $A_1$ satisfies the hypotheses that are required in Theorem \ref{ta1}.

We observe that if we introduce the function $z(t)=e^{\lambda_1 t}u(t)$, where $u$ is the solution of \eqref{stab}, then $z$ solves
\begin{equation}\label{z}
\left\{\begin{array}{ll}
z'(t)+A_1z(t)+p(t)Bz(t)=0,&t>0,\\
z(0)=u_0.
\end{array}\right.
\end{equation}
So, we can apply the previous analysis to this problem and deduce that there exist $M_T,\omega_T>0$ such that, for all $\rho>0$ there exists $R_\rho>0$ such that, if $||u_0-\varphi_1||\leq R_\rho$, then \begin{equation}\label{superex}
||z(t)-\varphi_1||\leq M_Te^{-\rho e^{\omega_T t}},\qquad \forall t\geq0.
\end{equation}

We claim that the local superexponetial stabilizability of $z$ to the stationary trajectory $\varphi_1$ implies the same property of $u$ to the ground state solution $\psi_1$. Indeed, it holds that
\begin{equation*}
||u(t)-\psi_1(t)||=||e^{-\lambda_1t}z(t)-e^{-\lambda_1t}\varphi_1||=e^{-\lambda_1t}||z(t)-\varphi_1||\leq M_Te^{-(\rho e^{\omega_T t}+\lambda_1t)},\quad \forall t\geq0
\end{equation*} 
and this concludes the proof also in the case of a strictly accretive operator $A$.
\end{proof}

\begin{oss}
Even in the case when $A:D(A)\subseteq X\to X$ has a finite number of negative eigenvalues, we can define the operator $A_1:=A-\lambda_1I$. $A_1$ has nonnegative eigenvalues and we can perform the proof of Theorem \ref{ta1} and deduce the superexponential stabilizability of the solution $u$ of the problem with diffusion operator $A$ to the ground state solution. In this case $\psi_1(t)=e^{\lambda_1t}\varphi_1$ blows up as $t\to\infty$ since $\lambda_1<0$, and the same occurs for the controlled solution $u$.
\end{oss}

\section{Applications}\label{examples}
In this section we discuss examples of bilinear control systems to which we can apply Theorem \ref{ta1}. The first problems we study are 1D parabolic equations of the form
\begin{equation*}
u_t(t,x)-u_{xx}(t,x)+p(t)Bu(t,x)=0,\quad(t,x)\in[0,T]\times(0,1)
\end{equation*}
in the state space $X=L^2(0,1)$, with Dirichlet or Neumann boundary conditions and with $B$ the following multiplication operators:
\begin{equation*}
Bu(t,x)=\mu(x)u(t,x).
\end{equation*}

Then, we prove the superexponential stabilizability of the following one dimensional equation with variable coefficients 
\begin{equation*}
u_t(t,x)-((1+x)^2u_x(t,x))_x+p(t)Bu(t,x)=0
\end{equation*}
with Dirichlet boundary condition.

Finally, we apply Theorem \ref{ta1} to the following parabolic equation 
\begin{equation*}
u_t(t,x)-\Delta u(t,x)+p(t)Bu(t,x)=0,\quad(t,x)\in[0,T]\times B^3
\end{equation*}
for radial data in the 3D unit ball $B^3$.

In each example, we will denote by $\{\lambda_k\}_{k\in\NN^*}$ and $\{\varphi_k\}_{k\in\NN^*}$, respectively, the eigenvalues and eigenfunctions of the second order operator associated with the problem under investigation. We will take $(\bar{u},\bar{p})=(\psi_1,0)$ as reference trajectory-control pair, where $\psi_1=e^{-\lambda_1 t}\varphi_1$ is the solution of the uncontrolled problem with initial condition $u(0,x)=\varphi_1$.

\subsection{Dirichlet boundary conditions.}\label{ex1}
Let $\Omega=(0,1)$, $X=L^2(\Omega)$ and consider the problem
\begin{equation}\label{1}\left\{\begin{array}{ll}
u_t(t,x)-u_{xx}(t,x)+p(t)\mu(x)u(t,x)=0 & x\in\Omega,t>0 \\
u=0 &x\in\partial \Omega, t>0\\
u(0,x)=u_0(x) & x\in\Omega,
\end{array}\right.
\end{equation}
where $p\in L^2(0,T)$ is the control function, $u$ the state variable, and $\mu$ is a function in $H^3(\Omega)$. 

We denote by $A$ the operator defined by
\begin{equation}\label{A}
D(A)=H^2\cap H^1_0(\Omega),\quad A\varphi=-\frac{d^2\varphi}{dx^2}.
\end{equation}
$A$ satisfies all the properties in \eqref{ip}: in particular, it is strictly accretive and its eigenvalues and eigenvectors have the following explicit expressions
\begin{equation}\nonumber
\lambda_k=(k\pi)^2,\quad \varphi_k(x)=\sqrt{2}\sin(k\pi x),\quad \forall k\in\NN^*.
\end{equation}
It is straightforward to prove that the eigenvalues fulfill the required gap property. Indeed, 
\begin{equation*}
\sqrt{\lambda_{k+1}}-\sqrt{\lambda_k}=(k+1)\pi-k\pi=\pi,\qquad \forall k\in \NN^*.
\end{equation*}
So, \eqref{gap} is satisfied.

In order to apply Theorem \ref{ta1} to system \eqref{1} and deduce the superexponential stabilizability to the trajectory $\psi_1$, we need to prove that there exists $\tau>0$ such that:
\begin{itemize}
\item $\langle B\varphi_1,\varphi_k\rangle\neq0$, for all $k\in\NN^*$,
\item the series
\begin{equation*}
\sum_{k\in\NN^*}\frac{e^{-2\lambda_k \tau}}{|\langle B\varphi_1,\varphi_k\rangle|^2}
\end{equation*}
is finite.
\end{itemize}
For this purpose, let us compute the scalar product $\langle B_0\varphi_1,\varphi_k\rangle=\langle\mu\varphi_1,\varphi_k\rangle$
\begin{equation*}\begin{split}
\langle \mu\varphi_1,\varphi_k\rangle&=\sqrt{2}\int_0^1 \mu(x)\varphi_1(x)\sin(k\pi x)dx\\
&=\sqrt{2}\left(-\left.(\mu(x)\varphi_1(x))\frac{\cos(k\pi x)}{k\pi}\right|^1_0+\int_0^1 (\mu(x)\varphi_1(x))'\frac{\cos(k\pi x)}{k\pi}dx\right)\\
&=\sqrt{2}\left(\left.(\mu(x)\varphi_1(x))'\frac{\sin(k\pi x)}{(k \pi)^2}\right|^1_0-
\int_0^1 (\mu(x)\varphi_1(x))''\frac{\sin(k\pi x)}{(k\pi)^2}dx\right)\\
&=\sqrt{2}\left(\left.(\mu(x)\varphi_1(x))''\frac{\cos(k\pi x)}{(k\pi)^3}\right|^1_0-\int_0^1(\mu(x)\varphi_1(x))'''\frac{\cos(k\pi x)}{(k\pi)^3}dx\right)\\
&=\frac{4}{k^3\pi^2}\left[\mu'(1)(-1)^{k+1}-\mu'(0)\right]-\frac{\sqrt{2}}{(k\pi)^3}\int_0^1(\mu(x)\varphi_1(x))'''\cos(k\pi x)dx.
\end{split}
\end{equation*}
Observe that the last integral term above represents the $k^{th}$-Fourier coefficient of the integrable function $(\mu(x)\varphi_1(x))'''$ and thus, it converges to zero as $k$ goes to infinity. Therefore, if we assume
\begin{equation}\label{abc}
\mu'(1)\pm\mu'(0)\neq 0\quad\mbox{ and }\quad\langle\mu\varphi_1,\varphi_k\rangle\neq0\quad\forall k \in \NN^*
\end{equation}
then, we deduce that $\langle \mu\varphi_1,\varphi_k\rangle$ is of order $1/k^3$ as $k\to\infty$. 

\begin{oss}
\emph{An example of a function which satisfies \eqref{abc} is $\mu(x)=x^2$. Indeed, in this case }
\begin{equation*}
\langle x^2\varphi_1,\varphi_k\rangle=\left\{\begin{array}{ll}
\frac{4k(-1)^k}{(k^2-1)^2},&  k\geq2,\\\\
\frac{2\pi^2-3}{6\pi^2},&k=1
\end{array}\right.
\end{equation*}
\emph{and so $\langle x^2\varphi_1,\varphi_k\rangle\neq0$ for all $k\in\NN^*$ and furthermore}
\begin{equation*}
|\langle x^2\varphi_1,\varphi_k\rangle|\geq\frac{2\pi^2-3}{6\pi^2}\frac{1}{k^3}=\frac{\pi(2\pi^2-3)}{6}\frac{1}{\lambda_k^{3/2}},\quad\forall k\in \NN^*.
\end{equation*}
\end{oss}

We conclude that, under assumption \eqref{abc},
\begin{equation}\label{lb}
\exists\,\, C>0 \mbox{ such that } |\langle B\varphi_1,\varphi_k\rangle|\geq ck^{-3}=C\lambda_k^{-3/2},\quad\forall k\in\NN^*
\end{equation}
and thanks to the polynomial behavior of the bound, the series
\begin{equation*}
\sum_{k\in\NN^*}\frac{e^{-2\lambda_k \tau}}{|\langle B\varphi_1,\varphi_k\rangle|^2}
\end{equation*}
converges for all $\tau>0$.

Therefore, all the hypotheses of Theorem \ref{ta1} are satisfied and system \eqref{1} is superexponentially stabilizable to the trajectory $\psi_1$.

\begin{oss}
\emph{Assumption \eqref{lb} for problem \eqref{1} is not too restrictive. In fact, it is possible to prove that the set of functions in $H^3(\Omega)$ for which \eqref{lb} holds is dense in $H^3(\Omega)$. For a proof of this fact, see Appendix A in \cite{bl}.}
\end{oss}

\subsection{Neumann boundary conditions}\label{ex3}
Now we look at an example with Neumann boundary conditions: let $\Omega=(0,1)$ and consider the following bilinear stabilzability problem
\begin{equation}\label{40}
\left\{\begin{array}{ll}
u_t(t,x)-\partial^2_{x}u(t,x)+p(t)\mu(x)u(t,x)=0 & x\in \Omega,t>0 \\
u_x=0 &x\in \partial \Omega,t>0\\
u(0,x)=u_0(x). & x\in\Omega
\end{array}\right.
\end{equation}

Let $X=L^2(\Omega)$. When we rewrite \eqref{40} in abstract form, the operators $A$ and $B$ are defined by
\begin{equation}\nonumber
D(A)=\{ \varphi\in H^2(0,1): \varphi'=0\mbox{ on }\partial\Omega\},\quad A\varphi=-\frac{d^2\varphi}{dx^2}
\end{equation}
\begin{equation}\nonumber
D(B)=X,\quad B\varphi=\mu\varphi.
\end{equation}
where $\mu$ is a real-valued function in $H^2(\Omega)$.

Operator $A$ satisfies the assumptions in \eqref{ip} and it is possible to compute explicitly its eigenvalues and eigenvectors:
\begin{equation*}
\begin{array}{lll}
\lambda_0=0,&\varphi_0=1\\
\lambda_k=(k\pi)^2,& \varphi_k(x)=\sqrt{2}\cos(k\pi x),& \forall k\geq1.
\end{array}
\end{equation*}

Since the eigenvalues are the same of those in Example \ref{ex1} for $k\geq1$, the gap condition is satisfied for all $k\geq0$.

Let us compute the scalar product $\langle \mu\varphi_0,\varphi_k\rangle$ to find, if it is possible, a lower bound of the Fourier coefficients of $B\varphi_0$:
\begin{equation*}\begin{split}
\langle \mu\varphi_0,\varphi_k\rangle&=\sqrt{2}\int_0^1 \mu(x)\cos(k\pi x)dx\\
&=\sqrt{2}\left(\left.\mu(x)\frac{\sin(k\pi x)}{k\pi}\right|^1_0-\int_0^1\mu'(x)\frac{\sin(k\pi x)}{k\pi}dx\right)\\
&=\sqrt{2}\left(\left.\mu'(x)\frac{\cos(k\pi x)}{(k\pi)^2}\right|^1_0-\int^1_0\mu''(x)\frac{\cos(k\pi x)}{(k\pi)^2}dx\right)\\
&=\frac{\sqrt{2}}{(k\pi)^2}\left(\mu'(1)(-1)^k-\mu'(0)\right)-\frac{\sqrt{2}}{(k\pi)^2}\int^1_0\mu''(x)\cos(k\pi x)dx.
\end{split}
\end{equation*}
Thus, reasoning as Example \ref{ex1}, if $\langle B\varphi_0,\varphi_k\rangle\neq0$ $\forall k \in\NN$ and $\mu\rq{}(1)\pm\mu\rq{}(0)\neq 0$, then we have that 
\begin{equation}\label{lb3}
\exists\,\, C>0\mbox{ such that } |\langle B\varphi_0,\varphi_k\rangle|\geq Ck^{-2}=C\lambda_k^{-1},\quad \forall k\in\NN^*
\end{equation}
and therefore the series in \eqref{a2} is finite for all $\tau>0$.

\begin{oss}
\emph{An example of a suitable function $\mu$ for problem \eqref{40} that satisfies the above hypothesis, is $\mu(x)=x^2$, for which}
\begin{equation*}
\langle x^2\varphi_0,\varphi_k\rangle=\left\{\begin{array}{ll}
\frac{2\sqrt{2}(-1)^{k}}{(k\pi)^2},&k\geq1,\\\\
\frac{1}{3},&k=0.
\end{array}\right.
\end{equation*}
\end{oss}

Applying Theorem \ref{ta1}, it follows that system \eqref{40} is superexponentially stabillizable to $\psi_1$.

\subsection{Dirichlet boundary conditions, variable coefficients}\label{ex5}
In this example, we analyze the superexponential stabilizability of a parabolic equation in divergence form with nonconstant coefficients in the second order term. 

Let $\Omega=(0,1)$, $X=L^2(\Omega)$ and consider the problem
\begin{equation}\label{nc52}
\left\{
\begin{array}{ll}
u_t(t,x)-((1+x)^2u_x(t,x))_x+p(t)\mu(x)u(t,x)=0&x\in\Omega,t>0\\
u(t,0)=0,\quad u(t,1)=0,&t>0\\
u(0,x)=u_0(x)&x\in\Omega
\end{array}
\right.
\end{equation}
where $p\in L^2(0,T)$ is the control and $\mu$ is a function in $H^2(\Omega)$ with some properties to be specified later.

We denote by $A$ the operator
$$A:D(A)\subset X\to X,\qquad Au=-((1+x)^2u_x)_x$$
where $D(A)=H^2\cap H^1_0(\Omega)$ and it is possible to prove that $A$ satisfies the properties in \eqref{ip}. The eigenvalues and eigenvectors of $A$ are computed as follows
$$\lambda_k=\frac{1}{4}+\left(\frac{k\pi}{\ln2}\right)^2,\qquad\varphi_k=\sqrt{\frac{2}{\ln 2}}(1+x)^{-1/2}\sin\left(\frac{k\pi}{\ln2 }\ln(1+x)\right).$$
The gap condition holds true because
$$\sqrt{\lambda_{k+1}}-\sqrt{\lambda_k}\geq \frac{\pi}{\ln 2},\quad\forall k\in\NN^*.$$

Now, we check the hypotheses on the operator $B\varphi=\mu\varphi$ needed to apply Theorem \ref{ta1}. We recall that we want to prove that:
\begin{itemize}
\item $\langle B\varphi_1,\varphi_k\rangle\neq 0$, for all $k\in\NN^*$,
\item there exists $\tau>0$ such that the series
\begin{equation}\label{series}
\sum_{k\in\NN^*}\frac{e^{-2\lambda_k \tau}}{|\langle B\varphi_1,\varphi_k\rangle|^2}
\end{equation}
is finite.
\end{itemize}
Let us compute the Fourier coefficients of $B\varphi_1$:
\begin{equation*}
\footnotesize{\begin{split}
\langle \mu&\varphi_1,\varphi_k\rangle=\sqrt{\frac{2}{\ln 2}}\int_0^1\mu(x)\varphi_1(x)(1+x)^{-1/2}\sin\left(\frac{k\pi}{\ln 2}\ln(1+x)\right) dx\\
&=\sqrt{\frac{2}{\ln 2}}\frac{\ln2}{k\pi}\left(-\left.\mu(x)\varphi_1(x)(1+x)^{1/2}\cos\left(\frac{k\pi}{\ln 2}\ln(1+x)\right)\right|^1_0+\right.\\
&\qquad\quad\quad\left.+\int_0^1\left(\mu(x)\varphi_1(x)(1+x)^{1/2}\right)'\cos\left(\frac{k\pi}{\ln 2}\ln(1+x)\right)dx\right)\\
&=\sqrt{\frac{2}{\ln 2}}\left(\frac{\ln2}{k\pi}\right)^2\left(\left(\mu(x)\varphi_1(x)(1+x)^{1/2}\right)'(1+x)\left.\sin\left(\frac{k\pi}{\ln 2}\ln(1+x)\right)\right|^1_0\right.+\\
&\qquad\quad\quad-\left.\int_0^1\left(\left(\mu(x)\varphi_1(x)(1+x)^{1/2}\right)'(1+x)\right)'\sin\left(\frac{k\pi}{\ln 2}\ln(1+x)\right)dx\right)\\
&=\sqrt{\frac{2}{\ln 2}}\left(\frac{\ln2}{k\pi}\right)^3\left(\left(\left(\mu(x)\varphi_1(x)(1+x)^{1/2}\right)'(1+x)\right)'(1+x)\left.\cos\left(\frac{k\pi}{\ln 2}\ln(1+x)\right)\right|^1_0\right.+\\
&\qquad\quad\quad-\left.\int_0^1\left(\left(\left(\mu(x)\varphi_1(x)(1+x)^{1/2}\right)'(1+x)\right)'(1+x)\right)'\cos\left(\frac{k\pi}{\ln 2}\ln(1+x)\right)dx\right)\\
&=\sqrt{\frac{2}{\ln 2}}\left(\frac{\ln2}{k\pi}\right)^3\left(\sqrt{\frac{2}{\ln 2}}\frac{2\pi}{\ln2}\left(-2\mu'(1)(-1)^k-\mu'(0)\right)+\right.\\
&\qquad\quad\quad\left.-\int_0^1\left(\left(\left(\mu(x)\varphi_1(x)(1+x)^{1/2}\right)'(1+x)\right)'(1+x)\right)'\cos\left(\frac{k\pi}{\ln 2}\ln(1+x)\right)dx\right)
\end{split}}
\end{equation*}
Observe that, for the same reason of Example \ref{ex1}, if $2\mu'(1)\pm\mu'(0)\neq0$ and $\langle \mu\varphi_1,\varphi_k\rangle\neq0$, $\forall k \in\NN^*$ then, there exists a constant $C>0$ such that $|\langle B_0\varphi,\varphi_k\rangle|$ is bounded from below by $C\lambda_k^{-3/2}$, for all $k\in\NN^*$. Thus, series \eqref{series} is finite for all $\tau>0$.

\begin{oss}
\emph{As an example of a function $\mu$ that verifies the lower bound $|\langle B\varphi,\varphi_k\rangle|\geq C\lambda_k^{-3/2}$, one can consider again $\mu(x)=x$: indeed, it satisfies the sufficient condition $2\mu'(1)\pm\mu'(0)\neq0$ and the Fourier coefficients of $B\varphi_1=x\varphi_1$ are all different from zero:}
\small{\begin{equation*}
\begin{split}
\langle x&\varphi_1,\varphi_k\rangle=\\
&=\left\{\begin{array}{ll}
\frac{2(2(-1)^{k+1}-1)}{(k^2-1)^2\left(1+\frac{(k+1)^2\pi^2}{(\ln 2)^2}\right)\left(1+\frac{(k-1)^2\pi^2}{(\ln 2)^2}\right)}\left(4k^3+k+1+2k(k^2-1)^2\frac{\pi}{(\ln 2)^2}\right),&k\geq2\\\\
\frac{1}{\ln 2}\left(\frac{(1-\ln2)\left(\frac{2\pi}{\ln2}\right)^3-\frac{2\pi}{\ln2}}{1+\left(\frac{2\pi}{\ln2}\right)^3}\right),&k=1
\end{array}\right.
\end{split}
\end{equation*}}
\end{oss}

This concludes the verification of the hypotheses of Theorem \ref{ta1}, that imply the superexponential stabilizability of \eqref{nc52} to $\psi_1$.

\subsection{$3D$ ball with radial data}\label{ex6}
In this example we consider an evolution equation in the three dimensional unit ball $B^3$ for radial data. The bilinear stabilizability problem is the following
\begin{equation}\label{51}
\left\{\begin{array}{ll}
u_t(t,r)-\Delta u(t,r)+p(t)\mu(r)u(t,r)=0 & r\in[0,1], t>0 \\
u(t,1)=0,&t>0\\
u(0,r)=u_0(r) & r\in[0,1]
\end{array}\right.
\end{equation}
where the Laplacian in polar coordinates for radial data has the form
$$\Delta\varphi(r)=\partial^2_r \varphi(r)+\frac{2}{r}\partial_r\varphi(r).$$
The function $\mu$ is a radial function as well in the space $H^3_r(B^3)$, where the spaces $H^k_r(B^3)$ are defined as follows
$$X:=L^2_{r}(B^3)=\left\{\varphi\in L^2(B^3)\,|\, \exists \psi:\RR\to\RR, \varphi(x)=\psi(|x|)\right\}$$
$$H^k_r(B^3):=H^k(B^3)\cap L^2_{r}(B^3) .$$

The domain of the Dirichlet Laplacian $A:=-\Delta$ in $X$ is $D(A)=H^2_{r}\cap H^1_0(B^3)$. We observe that $A$ satisfies the hypotheses required to apply Theorem \ref{ta1}. We denote by $\{\lambda_k\}_{k\in\NN^*}$ and $\{\varphi_k\}_{k\in\NN^*}$ the families of eigenvalues and eigenvectors of $A$, $A\varphi_k=\lambda_k\varphi_k$, namely
\begin{equation}\label{ee}\varphi_k=\frac{\sin(k\pi r)}{\sqrt{2\pi}r},\qquad\lambda_k=(k\pi)^2
\end{equation}
$\forall k\in\NN^*$, see \cite{leb}, section 8.14. The family $\{\varphi_k\}_{k\in\NN^*}$ forms an orthonormal basis of $X$.

In order to prove a superexponential stabilizability result to the trajectory $\psi_1$, we need to verify the remaining hypotheses in Theorem \ref{ta1} regarding the gap condition of the eigenvalues of $A$ and the properties of the operator $B:X\mapsto X$, $B\varphi=\mu\varphi$.

Since the Laplacian in the $3D$ ball for radial data behaves as a one dimensional operator, the analysis is very similar to the previous cases. Indeed, since the eigenvalues of the operator $A$ are actually the same of the $1D$ Dirichlet Laplacian, we have 
$$\sqrt{\lambda_{k+1}}-\sqrt{\lambda_k}=\pi,\quad\forall k\in\NN^*.$$

In order to compute a suitable lower bound for the Fourier coefficients of $B\varphi_1$, we recall the following property of radial symmetric functions $f=f(r)$: the integral over the unit ball $B^n\subset \RR^n$ of $f=f(r)$ reduces to
\begin{equation}
\int_{B^n}fdV=|S^{n-1}|\int_0^1 f(r)r^{n-1}dr
\end{equation}
where $|S^{n-1}|$ is the measure of the surface of the sphere $S^{n-1}$.

Therefore,
\begin{equation}\begin{split}
\langle \mu\varphi_1,\varphi_k\rangle&=\int_{B^3}\frac{1}{2\pi}\mu(r)\frac{\sin(\pi r)}{r}\frac{\sin(k\pi r)}{r}dV\\
&=4\pi \int_0^1 \frac{1}{2\pi} \mu(r)\frac{\sin(\pi r)}{r}\frac{\sin(k\pi r)}{r}r^2dr\\
&=\int_0^1 2\mu(r)\sin(\pi r)\sin(k\pi r)dr\\
&=-\frac{4}{k^3\pi^2}\left(\mu'(1)(-1)^k+\mu'(0)\right)+\\
&\quad-\frac{2}{(k\pi)^3}\int_0^1\left(\mu(r)\sin(\pi r)\right)'''\cos(k\pi r)dr.
\end{split}
\end{equation}
Following the same argument as in Example \ref{ex1}, if all the coefficients $\langle \mu\varphi_1,\varphi_k\rangle$ are different from zero and, moreover, $\mu'(1)\pm\mu'(0)\neq 0$ then, there exists a constant $C>0$ such that
\begin{equation*}
|\langle \mu\varphi_1,\varphi_k\rangle|\geq C\lambda_k^{-3/2},\qquad\forall k \in \NN^*
\end{equation*}
and thus the series in \eqref{series} is finite also in this case, for all $\tau>0$.

\begin{oss}
\emph{An example of a function $\mu\in H^3_{r}(B^3)$ with the aforementioned properties is $\mu(r)=r^2$. In this case the Fourier coefficients of $B\varphi_1$ are defined by}
\begin{equation*}
\langle B\varphi_1,\varphi_k\rangle=\left\{\begin{array}{ll}
\frac{8(-1)^{k+1}k}{(k^2-1)^2\pi^2},&k\geq2\\\\
\frac{2\pi^2-3}{6\pi^2},&k=1
\end{array}\right.
\end{equation*}
\end{oss}

Finally, applying Theorem \ref{ta1}, we deduce that, fixed $T>0$, there exist constants $M_T,\omega_T>0$ such that, for all $\rho>0$, there exists $R_\rho>0$ such that, if the initial condition $u_0$ satisfies $||u_0-\varphi_1||\leq R_\rho$, then
\begin{equation*}
||u(t)-\psi_1(t)||\leq M_Te^{-(\rho e^{\omega_T t}+\pi^2t)},\qquad\forall t>0.
\end{equation*}
\section*{References}
\bibliographystyle{plain}
\bibliography{biblio}

\begin{thebibliography}{10}

\bibitem{bms}
J.M. Ball, J.E. Marsden, and M.~Slemrod.
\newblock Controllability for distributed bilinear systems.
\newblock {\em SIAM Journal on Control and Optimization}, 20(4):575--597, 1982.

\bibitem{b}
K.~Beauchard.
\newblock Local controllability and non-controllability for a 1d wave equation
  with bilinear control.
\newblock {\em Journal of Differential Equations}, 250(4):2064--2098, 2011.

\bibitem{bl}
K.~Beauchard and C.~Laurent.
\newblock Local controllability of 1d linear and nonlinear {S}chr{\"o}dinger
  equations with bilinear control.
\newblock {\em J. Math. Pures Appl.}, 94:520--554, 2010.

\bibitem{bd}
A.~Bensoussan, G.~Da~Prato, M.C. Delfour, and S.K. Mitter.
\newblock {\em Representation and control of infinite dimensional systems},
  volume~1.
\newblock Birkh{\"a}user Boston, 1992.

\bibitem{cfk}
P.~Cannarsa, G.~Floridia, and A.~Y. Khapalov.
\newblock Multiplicative controllability for semilinear reaction-diffusion
  equations with finitely many changes of sign.
\newblock {\em Journal de Math{\'e}matiques Pures et Appliqu{\'e}es},
  108(4):425--458, 2017.

\bibitem{ck}
P.~Cannarsa and A.Y. Khapalov.
\newblock Multiplicative controllability for reaction-diffusion equations with
  target states admitting finitely many changes of sign.
\newblock {\em Discrete Contin. Dyn. Syst. Ser. B}, 14:1293--1311, 2010.

\bibitem{cmv}
P.~Cannarsa, P.~Martinez, and J.~Vancostenoble.
\newblock The cost of controlling weakly degenerate parabolic equations by
  boundary controls.
\newblock {\em Mathematical Control \& Related Fields}, 7(2):171--211, 2017.

\bibitem{cmsb}
T.~Chambrion, P.~Mason, M.~Sigalotti, and U.~Boscain.
\newblock Controllability of the discrete-spectrum {S}chr{\"o}dinger equation
  driven by an external field.
\newblock {\em Annales de l'IHP Analyse non lin{\'e}aire}, 26(1):329--349,
  2009.

\bibitem{fr}
H.O. Fattorini and D.L. Russell.
\newblock Exact controllability theorems for linear parabolic equations in one
  space dimension.
\newblock {\em Archive for Rational Mechanics and Analysis}, 43(4):272--292,
  1971.

\bibitem{fgip}
E.~Fern{\'a}ndez-Cara, S.~Guerrero, O~Y. Imanuvilov, and J.-P. Puel.
\newblock Local exact controllability of the navier--stokes system.
\newblock {\em Journal de math{\'e}matiques pures et appliqu{\'e}es},
  83(12):1501--1542, 2004.

\bibitem{f}
G.~Floridia.
\newblock Approximate controllability for nonlinear degenerate parabolic
  problems with bilinear control.
\newblock {\em Journal of Differential Equations}, 257(9):3382--3422, 2014.

\bibitem{k2}
A.Y. Khapalov.
\newblock Global non-negative controllability of the semilinear parabolic
  equation governed by bilinear control.
\newblock {\em ESAIM: Control, Optimisation and Calculus of Variations},
  7:269--283, 2002.

\bibitem{k3}
A.Y. Khapalov.
\newblock On bilinear controllability of the parabolic equation with the
  reaction-diffusion term satisfying newton’s law.
\newblock {\em J. Comput. Appl. Math}, 21(1):275--297, 2002.

\bibitem{k}
A.Y. Khapalov.
\newblock {\em Controllability of partial differential equations governed by
  multiplicative controls}.
\newblock Springer, 2010.

\bibitem{leb}
N.N. Lebedev.
\newblock {\em Special functions and their applications. Revised English
  edition. Translated and edited by Richard A. Silverman}.
\newblock Prentice-Hall Inc., Englewood Cliffs, NJ, 1965.

\bibitem{lions1}
J.-L. Lions.
\newblock Contr{\^o}labilit{\'e} exacte, perturbations et stabilisation de
  syst{\`e}mes distribu{\'e}s. tome 1.
\newblock {\em RMA}, 8, 1988.

\bibitem{lions2}
J.L. Lions.
\newblock Contr{\^o}labilit{\'e} exacte, stabilisation et perturbations de
  systemes distribues. tome 2., 1988.

\bibitem{p}
A.~Pazy.
\newblock {\em Semigroups of linear operators and applications to partial
  differential equations}, volume~44.
\newblock Springer Science \& Business Media, 2012.

\end{thebibliography}
\end{document}